\documentclass[11pt]{article}

\usepackage{amssymb,amsmath,amsfonts,amsthm}
\usepackage{latexsym}
\usepackage{graphics}
\usepackage{indentfirst}
\usepackage{hyperref}
\usepackage[capitalise, noabbrev, nameinlink]{cleveref}
\usepackage{comment}
\usepackage[shortlabels]{enumitem}
\usepackage[english]{babel}
\usepackage{tikz}
\usetikzlibrary{fit,arrows.meta,backgrounds}
\usetikzlibrary{shapes,decorations,graphs,graphs.standard,quotes,backgrounds}
\usetikzlibrary{decorations.markings}
\usepackage{esdiff}
\usepackage{thmtools}
\usepackage{thm-restate}

\definecolor{azure(colorwheel)}{rgb}{0.0, 0.5, 1.0}
\definecolor{Pink}{RGB}{255, 105, 180}

\usetikzlibrary{arrows.meta}

\newtheorem*{thm*}{Theorem}
\newtheorem{thm}{Theorem}
\newtheorem{lem}[thm]{Lemma}
\newtheorem{pro}[thm]{Proposition}
\newtheorem{obs}[thm]{Observation}
\newtheorem{cor}[thm]{Corollary}
\newtheorem{conj}[thm]{Conjecture}

\newtheorem{ques}[thm]{Question}

\crefname{thm}{Theorem}{Theorems}
\crefname{lem}{Lemma}{Lemmas}

\newcommand{\N}{\mathbb{N}}

\newcommand{\E}{\mathbb{E}}

\allowdisplaybreaks

\begin{document}

\title{Enumerative Chromatic Choosability}

\author{Sarah Allred \thanks{Department of Mathematics and Statistics, University of South Alabama, Mobile, AL, USA (sarahallred@southalabama.edu)}
\and 
Jeffrey A. Mudrock \thanks{Department of Mathematics and Statistics, University of South Alabama, Mobile, AL, USA (mudrock@southalabama.edu)}}

\maketitle

\begin{abstract}
Chromatic-choosablility is a notion of fundamental importance in list coloring.  A graph is \emph{chromatic-choosable} when its chromatic number is equal to its list chromatic number.  In 1990, Kostochka and Sidorenko introduced the list color function of a graph $G$, denoted $P_{\ell}(G,m)$, which is the list analogue of the chromatic polynomial of $G$, $P(G,m)$.  It is known that for any graph $G$ there is a positive integer $k$ such that $P_{\ell}(G,m) = P(G,m)$ whenever $m \geq k$. In this paper, we study enumerative chromatic-choosability.  A graph $G$ is \emph{enumeratively chromatic-choosable} when $P_{\ell}(G,m) = P(G,m)$ whenever $m \in \N$.  We completely determine the graphs of chromatic number two that are enumeratively chromatic-choosable.  We construct examples of graphs that are chromatic-choosable but fail to be enumeratively-chromatic choosable, and finally, we explore a conjecture as to  whether for every graph $G$, there is a $p \in \N$ such that the join of $G$ and $K_p$ is enumeratively chromatic-choosable.  The techniques we use to prove results are diverse and include probabilistic ideas and ideas from DP (or correspondence)-coloring.
 
\medskip



\end{abstract}

\section{Introduction}\label{intro}
In this paper, all graphs are nonempty, finite, simple graphs.  Generally speaking, we follow West~\cite{W01} for terminology and notation.  The set of natural numbers is $\N = \{1,2,3, \ldots \}$.  For $m \in \N$, we write $[m]$ for the set $\{1, \ldots, m \}$.  The AM-GM inequality refers to the Arithmetic Mean-Geometric Mean inequality.  We write $K_{l,n}$ for complete bipartite graphs with partite sets of size $l$ and $n$.  For a path $P$, we write $P=v_1,~v_2,\dots ,~ v_n$ to mean $P$ is the path with vertex set $\{v_1,\dots,v_n\}$ and edge set $\{v_iv_{i+1}~:~1\le i<n\}$.  If $G$ and $H$ are vertex disjoint graphs, we write $G \vee H$ for the join of $G$ and $H$.

\subsection{List Coloring and Chromatic-Choosability} \label{basic}

For classical vertex coloring of graphs, we wish to color the vertices of a graph $G$ with up to $m$ colors from $[m]$ so that adjacent vertices receive different colors, a so-called \emph{proper $m$-coloring}.  The \emph{chromatic number} of a graph is a very well studied graph invariant. Denoted $\chi(G)$, the chromatic number of graph $G$ is the smallest $m$ such that $G$ has a proper $m$-coloring.  List coloring is a well-known variation on classical vertex coloring which was introduced independently by Vizing~\cite{V76} and Erd\H{o}s, Rubin, and Taylor~\cite{ET79} in the 1970s.  For list coloring, we associate a \emph{list assignment} $L$ with a graph $G$ such that each vertex $v \in V(G)$ is assigned a list of available colors $L(v)$ (we say $L$ is a list assignment for $G$).  We say $G$ is \emph{$L$-colorable} if there is a proper coloring $f$ of $G$ such that $f(v) \in L(v)$ for each $v \in V(G)$ (we refer to $f$ as a \emph{proper $L$-coloring} of $G$).  A list assignment $L$ for $G$ is called a \emph{$m$-assignment} if $|L(v)|=m$ for each $v \in V(G)$.  Graph $G$ is said to be $m$-choosable if it is $L$-colorable whenever $L$ is a $m$-assignment for $G$.  The \emph{list chromatic number} of a graph $G$, denoted $\chi_\ell(G)$, is the smallest $k$ for which $G$ is $m$-choosable.  It is easy to show that for any graph $G$, $\chi(G) \leq \chi_\ell(G)$.  Moreover, the gap between the chromatic number and list chromatic number of a graph can be arbitrarily large since $\chi_{\ell}(K_{n,t}) = n+1$ whenever $t \geq n^n$ (see e.g.,~\cite{M18} for further details).  

A graph $G$ is called \emph{chromatic-choosable} if $\chi_{\ell}(G) = \chi(G)$~\cite{O02}.   Determining whether a graph is chromatic-choosable is, in general, a challenging problem.  Perhaps the most well known conjecture involving list coloring is about chromatic-choosability.  Indeed, the famous Edge List Coloring Conjecture states that every line graph of a loopless multigraph is chromatic-choosable (see~\cite{HC92}).

It is shown in~\cite{O02} that for any graph $G$, there exists an $N \in \N$ such that $K_p \vee G$ is chromatic-choosable whenever $p \geq N$.  In 2015, Noel, Reed, and Wu~\cite{NR15} generalized this by famously proving the following.

\begin{thm} [\cite{NR15}] \label{thm: Ohba}
If $G$ is a graph satisfying $|V(G)| \leq 2 \chi(G) + 1$, then $G$ is chromatic-choosable.
\end{thm}

\subsection{Counting Proper Colorings and List Colorings}

In 1912, Birkhoff~\cite{B12} introduced the notion of the chromatic polynomial in hopes of using it to make progress on the four color problem.  For $m \in \N$, the \emph{chromatic polynomial} of a graph $G$, $P(G,m)$, is the number of proper $m$-colorings of $G$. It is easy to show that $P(G,m)$ is a polynomial in $m$ of degree $|V(G)|$ (e.g., see~\cite{DKT05}). For example, whenever $n \in \N$ it is well known that $P(K_n,m) = \prod_{i=0}^{n-1} (m-i)$, $P(C_n,m) = (m-1)^n + (-1)^n (m-1)$, $P(T,m) = m(m-1)^{n-1}$ whenever $T$ is a tree on $n$ vertices, $P(K_{2,n},m) = m(m-1)^n + m(m-1)(m-2)^n$, and $P(K_1 \vee G, m) = m P(G,m-1)$ for each $m \geq 2$ (see~\cite{B94} and~\cite{W01}).  Note that the formula for the chromatic polynomial of a cycle on $n$ vertices still holds when it is not simple (i.e., when $n=1,2$). 

The notion of chromatic polynomial was extended to list coloring in the early 1990s by Kostochka and Sidorenko~\cite{AS90}.  If $L$ is a list assignment for $G$, let $P(G,L)$ denote the number of proper $L$-colorings of $G$. The \emph{list color function} $P_\ell(G,m)$ is the minimum value of $P(G,L)$ where the minimum is taken over all possible $m$-assignments $L$ for $G$.  Since an $m$-assignment could assign the same $m$ colors to every vertex in a graph, it is clear that $P_\ell(G,m) \leq P(G,m)$ for each $m \in \N$.  In general, the list color function can differ significantly from the chromatic polynomial for small values of $m$.  For example, for any $n \geq 2$, $P_{\ell}(K_{n,n^n},m) = 0$ and $P(K_{n,n^n},m) > 1$ whenever $m \in \{2, \ldots, n\}$.  On the other hand,  in 1992, Donner~\cite{D92} showed that for any graph $G$ there is a $k \in \N$ such that $P_\ell(G,m) = P(G,m)$ whenever $m \geq k$. 

\subsection{Enumeratively Chromatic-Choosable Graphs} 

With Donner's 1992 result in mind, it is natural to study the point at which the list color function of a graph becomes identical to its chromatic polynomial.  Given any graph $G$, the \emph{list color function number of $G$}, denoted $\nu(G)$, is the smallest $t \geq \chi(G)$ such that $P_{\ell}(G,t) = P(G,t)$.  The \emph{list color function threshold of $G$}, denoted $\tau(G)$, is the smallest $k \geq \chi(G)$ such that $P_{\ell}(G,m) = P(G,m)$ whenever $m \geq k$.  Clearly, $\chi(G) \leq \chi_\ell(G) \leq \nu(G) \leq \tau(G)$.

One of the most important open questions on the list color function asks whether the list color function number of a graph can differ from its list color function threshold. 

\begin{ques} [\cite{KN16}] \label{ques: threshold}
	For every graph $G$, is it the case that $\nu(G) = \tau(G)$?
\end{ques} 

Much of the research on the list color function has been focused on studying the list color function threshold for a general graph $G$.  Currently, the best known bound is as follows.

\begin{thm} [\cite{DZ22}] \label{thm: thresh}
For every graph $G$ with $|E(G)| \geq 4$, $\tau(G) \leq |E(G)|-1$.
\end{thm}

   As far as we are aware the notion of enumerative chromatic choosability was fist formally defined in~\cite{KK23} even though it has been pursued since the 1990s.  A graph $G$ is said to be \emph{enumeratively chromatic-choosable} if $\tau(G) = \chi(G)$.
    This means that a graph is \emph{enumeratively chromatic-choosable} if $P_{\ell}(G,m)=P(G,m)$ for $m\ge\chi(G)$.
  Additionally, we call a graph \emph{weakly enumeratively chromatic-choosable} if $\nu(G) = \chi(G)$; that is, when $P_{\ell}(G,\chi(G))=P(G,\chi(G))$.
  We now present some results on these notions.  Note that  \emph{generalized theta graphs} $\Theta(l_1, \ldots, l_n)$ consist of a pair of end vertices joined by $n$ internally disjoint paths of lengths $l_1, \ldots, l_n \in \N$. When there are three such paths we have a \emph{theta graph}. 

\begin{thm} \label{thm: examples} [\cite{BK23, HK23, KK23, KM18, KN16, AS90}]
The following statements hold.

\begin{enumerate}[label=(\roman*)]
    \item Chordal graphs are enumeratively chromatic-choosable.
    \item Cycles are enumeratively chromatic-choosable
    \item $K_{2,3}$ is enumeratively chromatic-choosable.
    \item If $1 \leq l_1 \leq l_2$, $2 \leq l_2 \leq l_3$, and the parity of $l_1$ is different from both $l_2$ and $l_3$, then $\Theta(l_1, l_2, l_3)$ is enumeratively chromatic-choosable.
    \item If $G$ is enumeratively chromatic-choosable (resp. weakly enumeratively \\ chromatic-choosable), then $G \vee K_n$ is enumeratively chromatic-choosable (resp. weakly enumeratively chromatic-choosable) for any $n \in \N$.  
\end{enumerate}
\end{thm}

Statement~(v) is a consequence of the following more general result.

\begin{thm}[\cite{KM18}]\label{thm: joinK1}
Let $G$ be a graph such that $P_{\ell}(G,m)=P(G,m)$. Then $P_{\ell}(G\vee K_1,m+1)=P(G\vee K_1,m+1)$. 
\end{thm}

With this and \cref{thm: Ohba} in mind we conjecture the following.

\begin{conj} \label{conj: join}
For any graph $G$, there is a $p \in \N$ such that $G \vee K_p$ is enumeratively chromatic-choosable.    
\end{conj}

Any graph that is weakly enumeratively chromatic-choosable is chromatic-choosable.  Interestingly, there are few known examples of graphs that are chromatic-choosable and not weakly enumeratively chromatic-choosable.  In 2009, Carsten Thomassen showed the following which is illustrated in~\cref{fig:t224}.
\begin{pro}\label{lem: Theta224cc}\cite{T09}
 $\Theta(2,2,4)$ is chromatic-choosable, but it is not weakly enumeratively chromatic-choosable.
\end{pro}

\begin{figure}[htb]
    \begin{center}
     \begin{tikzpicture}
        [scale=1,auto=left,every node/.style={circle, fill, inner sep=0 pt, minimum size=2mm, outer sep=0pt},line width=.4mm]
        \node (u) at (0,1.75)[label={[label distance=2pt]left:{$\{1,3\}$}}] {};
        \node (v) at (4,1.75)[label={[label distance=2pt]right:{$\{1,2\}$}}]{};
        \node (1) at (2,3.5)[label={[label distance=-7pt]above:{$\{1,2\}$}}]{};
        \node (2) at (2,1.75)[label={[label distance=-7pt]above:{$\{2,3\}$}}]{};
        \node (3) at (1,0) [label={[label distance=-7pt]below:{$\{1,3\}$}}]{};
        \node (4) at (2,0)[label={[label distance=-7pt]below:{$\{2,3\}$}}]{};
        \node (5) at (3,0)[label={[label distance=-7pt]below:{$\{1,2\}$}}]{};
        \begin{scope}[on background layer]
        \draw[line width=.4mm] (u.center) to (1.center) to (v.center) to (5.center) to (4.center) to (3.center) to (u.center) to (2.center) to (v.center);
        \end{scope}
        \end{tikzpicture}
        \end{center}
        \caption{A 2-assignment for a copy of $\Theta(2,2,4)$ with exactly one proper coloring}
        \label{fig:t224}
\end{figure}
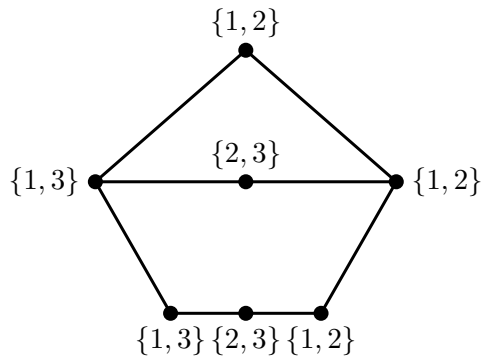

In \cref{cor: Theta222tcc}, we generalize this to $\Theta(2,2,2k)$, and in \cref{pro: counterexample} we show that while $K_{2,2,4}$ is chromatic-choosable, it is not weakly enumeratively chromatic-choosable.  In general, it is challenging to determine which graphs are enumeratively chromatic-choosable and few such graphs are known.  Consequently, we will keep the following two questions in mind for the remainder of the paper.

\begin{ques} \label{ques: ecc}
Which graphs are enumeratively chromatic-choosable?
\end{ques}  

\begin{ques} \label{ques: notecc}
For each $k \geq 2$, does there exist a chromatic-choosable graph $G$ such that $\chi(G) = k$ and $G$ is not weakly enumeratively chromatic-choosable?
\end{ques} 

\cref{cor: Theta222tcc} and \cref{pro: counterexample} imply that the answer to \cref{ques: notecc} is yes when $k=2$ and $k=3$.

\subsection{Outline of Paper}

We now present an outline of the paper while mentioning some open questions.  In \cref{char}, we completely characterize the enumeratively chromatic-choosable graphs with chromatic number 2.  
Following~\cite{ET79}, we say that the \emph{core} of a connected graph $G$ is the graph obtained from $G$ by successively deleting vertices of degree 1.
Note that the core of a tree is a $K_1$.

\begin{restatable}{thm}{thmcharacterize} \label{thm: characterize}
Suppose $G$ is a connected graph with $\chi(G)=2$.  We have that $G$ is enumeratively chromatic-choosable if and only if the core of $G$ is a copy of: $K_1$, $C_{2k+2}$ for $k \in \N$, or $K_{2,3}$.
\end{restatable}

Interestingly, the enumeratively chromatic-choosable graphs with chromatic number 2 are the same as the chromatic-paintable graphs with chromatic number 2 (see~\cite{Z09}).~\footnote{Chromatic-paintability is an online version of list coloring that was introduced by Xuding Zhu~\cite{Z09}.}

In \cref{theta} we consider the list color function of $G = \Theta(2,2,2k)$ where $k \geq 2$.  \cref{cor: Theta222tcc} tells us $G$ is chromatic-choosable, but not weakly chromatic-choosable.  Also, recall that for any graph $G$, there is an $N \in \N$ such that $G \vee K_N$ is chromatic-choosable.  Consequently, it is natural to consider \cref{conj: join} when $G = \Theta(2,2,2k)$ where $k \geq 2$.  We use the AM-GM Inequality and DP-coloring to prove the following theorem that such a $G$ is close to being enumeratively chromatic-choosable. 

\begin{restatable}{thm}{thmthreeandabove} \label{thm: 3andabove}
Suppose $G = \Theta(2,2,2k)$ with $k \geq 2$.  Then, $P_{\ell}(G,m) = P(G,m)$ whenever $m \geq 3$.
\end{restatable}

Then, in \cref{prob}, we show that $K_1 \vee \Theta(2,2,2k)$ is enumeratively chromatic-choosable (\cref{thm: thetagraph}) with the help of a probabilistic lemma. This provides some support for \cref{conj: join}. 
 We also show that while $K_{2,2,4}$ is chromatic-choosable, it fails to be enumeratively chromatic-choosable (\cref{pro: counterexample}) which indicates that it would be natural to study \cref{conj: join} when $G$ is a complete multipartite graph.  This is a topic for further research. 

\section{Enumeratively Chromatic-Choosable Graphs with $\chi(G)=2$} \label{char}

In this section, it is useful to keep in mind that if $G$ is a connected graph with $\chi(G)=2$, then $G$ has unique bipartition which means $P(G,2) = 2$.  In \cite{ET79}, Erd\"{o}s, Rubin, and Taylor proved the following.

\begin{thm}\label{thm: 2choosable}\cite{ET79}
Let $G$ be a connected bipartite graph.  Then $G$ is 2-choosable if and only if the core of $G$ belongs to $K_1$, $C_{2k+2}$ for $k\in \N$, or $\Theta(2,2,2k)$ for $k\ge 1$.
\end{thm}

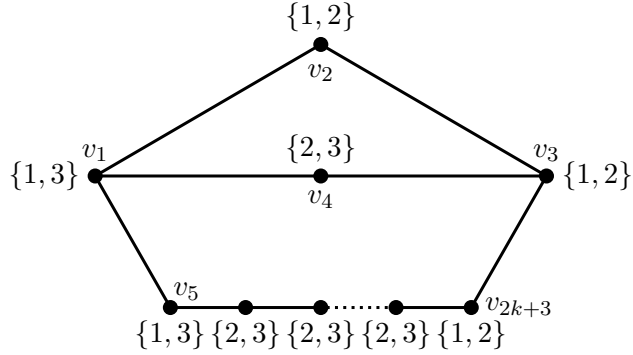
\begin{figure}[htb]
\begin{center}
     \begin{tikzpicture}
        [scale=1,auto=left,every node/.style={circle, fill, inner sep=0 pt, minimum size=2mm, outer sep=0pt},line width=.4mm]
        \node (u) at (0,1.75)[label={[label distance=2pt]left:{$\{1,3\}$}}, label=above:$v_1$] {};
        \node (v) at (6,1.75)[label={[label distance=2pt]right:{$\{1,2\}$}},label=above: $v_3$]{};
        \node (1) at (3,3.5)[label={[label distance=-7pt]above:{$\{1,2\}$}}, label={[label distance=2pt]below:$v_2$}]{};
        \node (2) at (3,1.75)[label={[label distance=-7pt]above:{$\{2,3\}$}}, label=below:$v_4$]{};
        \node (3) at (1,0) [label={[label distance=-7pt]below:{$\{1,3\}$}}, label=above right:$v_5$]{};
        \node (4) at (2,0)[label={[label distance=-7pt]below:{$\{2,3\}$}}]{};
        \node (5) at (3,0)[label={[label distance=-7pt]below:{$\{2,3\}$}}]{};
        \node (6) at (4,0)[label={[label distance=-7pt]below:{$\{2,3\}$}}]{};
        \node (7) at (5,0) [label={[label distance=2pt]right:$v_{2k+3}$},label={[label distance=-7pt]below:{$\{1,2\}$}}]{};
        \begin{scope}[on background layer]
        \draw[line width=.4mm] (u.center) to (1.center) to (v.center) to (7.center) to (6.center); 
        \draw[line width=.4mm, dotted] (6.center) to (5.center);
        \draw[line width=.4mm] (5.center) to (4.center) to (3.center) to (u.center) to (2.center) to (v.center);
        \end{scope}
        \end{tikzpicture}
        \end{center}
        \caption{A 2-assignment for a copy of $\Theta(2,2,2k)$ with exactly one proper coloring}
        \label{fig:k222t}
\end{figure}

\begin{lem}\label{cor: Theta222tcc}
For each $k \ge 2$, $\Theta(2,2,2k)$ is chromatic-choosable, but not weakly enumeratively chromatic-choosable.
\end{lem}
\begin{proof}
Let $G=\Theta(2,2,2k)$ for some $k\ge 2$.  
By \cref{thm: 2choosable}, it follows that $\chi_{\ell}(G)=\chi(G)=2$.  
So, $G$ is chromatic-choosable.
Let the three paths of $G$ be $S_1=v_1,v_2,v_3$,  $S_2=v_1,v_4,v_3$, and $S_3=v_1,v_5, v_6,\dots, v_{2k+3}, v_3$.
If $k=2$, the desired result follows by \cref{lem: Theta224cc}.  We may therefore assume that $k\ge 3$.
Let $L$ be the $2$-assignment for $G$ be defined by
$L(v_1)=\{1,3\}$, $L(v_2)=\{1,2\}$, $L(v_3)=\{1,2\}$, $L(v_4)=\{2,3\}$, $L(v_5)=\{1,3\}$, $L(v_i)=\{2,3\}$ for $6\le i \le 2k+2$, and $L(v_{2k+3})=\{1,2\}$. 

Let $f$ be the $L$-coloring given by $f(v_1)=3$, $f(v_4)=2$, $f(v_3)=1$, $f(v_2)=2$, $f(v_5)=1$, $f(v_{2k+3})=2$ $f(v_i)=3$ for even $i\in \{6,7,\dots, 2k+2\}$, and $f(v_j)=2$ for odd $j\in\{6,7,\dots, 2k+2\}$.
Clearly, this is a proper $L$-coloring.

We claim that $f$ is the only proper $L$-coloring of $G$.
Suppose for a contradiction that there is a proper $L$-coloring $g$ distinct from $f$.
We claim that $g(v_1)=3$.  To see why, suppose for a contradiction that $g(v_1)=1$. 
Then $g(v_2)=2$, $g(v_3)=1$, $g(v_{2k+3})=2$, $g(v_i)=3$ for even $i\in\{6,7,\dots, 2k+2\}$ and $g(v_j)=2$ for odd $j\in \{6,7,\dots, 2k+2\}$.
Since $g(v_6)=3$, $g(v_1)=1$, and $g(v_5)\in\{1,3\}$, it follows that $g$ is not a proper $L$-coloring. 
So $g(v_1)=3=f(v_1)$, and
it is easy to see this implies $g=f$ which is a contradiction.

Thus, $P_{\ell}(G,2) \leq P(G,L) = 1 < 2 = P(G,2)$.  The conclusion follows.
\end{proof}

Suppose that $G$ is a graph and $z \in V(G)$ has degree 1.  It is well-known that $P(G,m) = (m-1)P(G-z,m)$ for each $m \in \N$ (see~\cite{W01}).  We will now show that an analogous identity holds for the list color function.

\begin{lem} \label{lem: pendant}
Suppose $G$ is a graph and $m \in \N$.  Suppose $G'$ is the graph obtained from $G$ by adding a new vertex $z$ to $G$ and connecting $z$ to exactly one vertex in $G$.  Then, $P_{\ell}(G',m) = (m-1)P_{\ell}(G,m)$.
\end{lem}

\begin{proof}
Suppose $x \in V(G)$ is the vertex to which $z$ is adjacent. Suppose $L$ is an $m$-assignment such that $P_{\ell}(G,m) = P(G,L)$.  Let $L'$ be the $m$-assignment for $G'$ given by $L'(v)=L(v)$ for each $v \in V(G)$ and $L'(z)=L(x)$.  It is clear that each proper $L$-coloring of $G$ can be extended to a proper $L'$-coloring of $G'$ in $(m-1)$ ways.  Consequently,
\[P_{\ell}(G',m) \leq P(G',L') = (m-1)P(G,L) = (m-1)P_{\ell}(G,m).\]
Now, suppose that $K'$ is an $m$-assignment for $G'$ such that $P_{\ell}(G',m) = P(G',K')$.  Let $K$ be the $m$-assignment for $G$ obtained by restricting the domain of $K'$ to $V(G)$.  Now, each proper $K$-coloring of $G$ can be extended to a proper $K'$-coloring of $G'$ in $m$ or $(m-1)$ ways (depending on whether the color used on $x$ is in $K'(z)$).  Thus,
\[P_{\ell}(G',m) \geq P(G',K') \geq (m-1)P(G,K) = (m-1)P_{\ell}(G,m).\]
The result follows.
\end{proof}

\thmcharacterize*
\begin{proof}
    Suppose the core of $G$ is a copy of $K_1$, $C_{2k+2}$ for $k\in \N$, or $K_{2,3}$.  
    We will proceed by induction on the number of vertices of $G$, say $D$, deleted to obtain the core.
    If $G$ required zero deletions, then \cref{thm: examples} implies that $G$ is enumeratively chromatic-choosable. 
    Now, assume that $D\ge 1$ and any graph with core $K_1$, $C_{2k+2}$ for $k \in \N$, or $K_{2,3}$ resulting from less than $D$ vertex deletions is enumeratively chromatic-choosable.
    Let $z\in V(G)$ satisfy $\deg_G(z)=1$, and let $G'=G-z$. 
    By the induction hypothesis, we have that $P_{\ell}(G',m)=P(G',m)$ for $m\ge 2$.  Notice 
    \begin{equation*}
        \begin{split}
            P_{\ell}(G,m)&=(m-1)P_{\ell}(G',m) \text{ by \cref{lem: pendant}}\\
            &=(m-1)P(G',m) \text{ by the induction hypothesis}\\
            &=P(G,m)
        \end{split}
    \end{equation*}
Thus, $G$ is enumeratively chromatic-choosable.

    Conversely, suppose that $G$ is enumeratively chromatic-choosable.
    Since $G$ is enumeratively chromatic-choosable, it is chromatic-choosable.
    Since $\chi(G)=2$, it follows that $G$ is bipartite and 2-choosable.
    \cref{thm: 2choosable} implies that the core of $G$ belongs to $K_1$, $C_{2k+2}$ for $k\in \N$, or $\Theta(2,2,2k)$ for $k\ge 1$.
    If the core of $G$ is one of the following $K_1$, $C_{2k+2}$  for $k \in \N$, or $\Theta(2,2,2)=K_{2,3}$, then the conclusion follows. 
    Now, suppose that the core, $C$, of $G$ is $\Theta(2,2,2k)$ for $k\ge 2$.  We will derive a contradiction in this case. Let the vertices of the core be $w_1$, $w_2$, \dots, $w_{2k+3}$ where the paths of the core are $S_1=w_1,w_2,w_3$, $S_2=w_1,w_4,w_3$, and $S_3=w_1,w_5,w_6,\dots ,w_{2k+3},w_3$. 

    If $C=G$, then we have a contradiction by \cref{cor: Theta222tcc}.  So, we may assume that $G \neq C$.  Let $\{v_1, \ldots, v_D\} = V(G) - V(C)$.  Additionally, suppose the degree of $v_1$ is 1 in $G$, and if $D \geq 2$, the degree of $v_{i+1}$ in $G - \{v_1, \ldots, v_i\}$ is 1 for each $i \in [D-1]$. For each $i \in [D]$ let $G_i = G - \{v_1,\dots, v_{i}\}$.  Also, let $G_0 = G$.
    Let $L$ be the $2$-assignment for $G$ defined as follows.  First, let
    $L(w_1)=\{1,3\}$, $L(w_2)=\{1,2\}$, $L(w_3)=\{1,2\}$, $L(w_4)=\{2,3\}$, $L(w_5)=\{1,3\}$, $L(w_i)=\{2,3\}$ for $6\le i \le 2k+2$, and $L(w_{2k+3})=\{1,2\}$, as in \cref{cor: Theta222tcc}.
    We will give the definition for $L(v_1),\ldots, L(v_D)$ in reverse order.
    Note that in $G_{i-1}$, the vertex $v_i$ is a leaf for $i \in [D]$.
    For $i=D,\dots,1$, let $L(v_i)$ be the same as the list of its neighbor in $G_{i-1}$.  This completes the definition of $L$.
    
    Suppose $f$ and $g$ are proper $L$-colorings of $G$.  \cref{cor: Theta222tcc} implies if $f$ and $g$ disagree, it is at one of the vertices in $\{v_1,\dots,v_D \}$.  Assume $j$ is the largest index such that $f(v_j)\ne g(v_j)$.  The vertex $v_j$ has a unique neighbor, say $z$ in $G_{j-1}$.  It follows that either $z \in V(C)$ or $z=v_k$ for some $k \in \{j+1, \ldots, D \}$. Consequently, our assumption implies that $f(z)=g(z)$.  So $ g(v_j), f(v_j)\in L(v_j)-\{f(z)\}$.  Note that $L(v_j)-\{f(z)\}$ only has one element.  Thus $f(v_j)=g(v_j)$.  This is a contradiction.  Thus there is only one proper $L$-coloring of $G$ and $P_{\ell}(G,2)=1\le P(G,2)$ which contradicts the fact that $G$ is enumeratively chromatic-choosable. 
    Thus, the core of $G$ is one of the following: $K_1$, $C_{2k+2}$ for $k\in\N$, and $K_{2,3}$, as required.
\end{proof}

Interestingly, the enumeratively chromatic-choosable graphs with chromatic number 2 are the same as the chromatic-paintable graphs with chromatic number 2.

\begin{thm} [\cite{Z09}] \label{thm: 2paintability}
A connected graph $G$ is 2-paintable if and only if the core of $G$ is a copy of: $K_1$, $C_{2k+2}$ for $k\in \N$, or $K_{2,3}$.
\end{thm}

\section{2-Choosable Theta Graphs} \label{theta}

In this section we prove \cref{thm: 3andabove}; throughout, unless otherwise noted, suppose that $G = \Theta(2,2,2k)$ where $k \geq 2$.  Suppose that the endpoints of the paths used to construct $G$ are $u$ and $v$.  Suppose the internal vertices of the paths of length 2 used to construct $G$ are $x_1$ and $y_1$.  We call these paths $S_1$ and $S_2$ respectively.  Suppose the internal vertices of the path of length $2k$ used to construct $G$ are $z_1, \ldots, z_{2k-1}$.  We call this path $S_3$.  When $L$ is an $m$-assignment for $G$, for each $i \in [3]$, let $L_i$ be the $m$-assignment for $S_i$ obtained by restricting the domain of $L$ to $V(S_i)$.  For each $i \in [3]$, $c \in L(u)$, and $d \in L(v)$, let $N_i(c,d)$ be the number of proper $L_i$-colorings of $S_i$ where $u$ is colored with $c$ and $v$ is colored with $d$.  We begin with an important lemma.

\begin{lem} \label{lem: sum}
Suppose that $L$ is an $m$-assignment for $G$.  Then,
\[P(G,L) = \sum_{(c,d) \in L(u) \times L(v)} \prod_{i=1}^3 N_i(c,d).\]
\end{lem}

\begin{proof}
    For each $i\in[3]$, let $C^{(i)}_{(c,d)}$ be the set of proper $L_i$-colorings of $S_i$ where $u$ is colored with $c$ and $v$ is colored with $d$.    
    
    Let $M: \bigcup\limits_{(c,d)\in L(u)\times L(v)} \left(C^{(1)}_{(c,d)}\times C^{(2)}_{(c,d)}\times C^{(3)}_{(c,d)}\right)\rightarrow \{f: \text{$f$ is a proper }L\text{-coloring of }G\}$ be the function defined by $M\left((g_1,g_2,g_3)\right)=T$ where $T:V(G)\rightarrow \bigcup_{w \in V(G)} L(w)$ is given by $T(u)=g_1(u)$, $T(v)=g_1(v)$, $T(x_1)=g_1(x_1)$, $T(y_1)=g_2(y_1)$, and $T(z_i)=g_3(z_i)$ for $i \in [2k-1]$.  Clearly, $T$ is a proper $L$-coloring of $G$.  It is easy to verify that $M$ is a bijection.

    Since we are considering a Cartesian product, it follows that $\left|C_{(c,d)}^{(1)}\times C_{(c,d)}^{(2)}\times C_{(c,d)}^{(3)}\right|=\prod_{i=1}^3 N_i(c,d)$.  By the addition rule, we
    have that \[\left|\bigcup\limits_{(c,d)\in L(u)\times L(v)} \left(C^{(1)}_{(c,d)}\times C^{(2)}_{(c,d)}\times C^{(3)}_{(c,d)}\right) \right|=\sum_{(c,d) \in L(u) \times L(v)} \prod_{i=1}^3 N_i(c,d).\]
    Since $M$ is a bijection, this is the number of proper $L$-colorings of $G$, as required.
\end{proof}

The AM-GM inequality immediately yields the following.
\begin{cor}\label{cor: amgmpgl}
    Suppose that $L$ is an $m$-assignment for $G$.   Then,
    \[P(G,L)=\sum\limits_{(c,d) \in L(u) \times L(v)} \prod\limits_{i=1}^3 N_i(c,d)\ge m^2\left(\prod_{(c,d) \in L(u) \times L(v)} \prod_{i=1}^3 N_i(c,d)\right)^{1/m^2}.\]
\end{cor}

We now prove a series of lemmas that will culminate in the proof of \cref{thm: 3andabove}.  Our general strategy is to show that for any $m$-assignment $L$ for $G$ with $m \geq 3$, $P(G,L) \geq P(G,3)$.  We prove this in the case that $L(u)= L(v)$, the case that $m=3$ and $L(u) \neq L(v)$, and the case that $m > 3$ and $L(u) \neq L(v)$.  Importantly, the formula for the chromatic polynomial of theta graphs is well known.  In particular, we have that for each $m \in \N$, 
\[P(G,m) = (m-2)^2 ((m-1)^{2k+1}-(m-1))+ (m-1)^2((m-1)^{2k} + (m-1)) \]
(see formula (2.6a) in~\cite{Bh01}).

\begin{lem} \label{lem: samelist}
Suppose that $L$ is an $m$-assignment for $G$ such that $m \geq 3$ and $L(u)=L(v)$.  Then, $P(G,L) \geq P(G,m)$.
\end{lem}

\begin{proof}
Without loss of generality assume that $L(u) = L(v) = [m]$. 
 Suppose $(c,d) \in [m]^2$. For $i \in [2]$, we have that $N_i(c,d) \geq m-2$ when $c \neq d$, and we have that $N_i(c,d) \geq m-1$ when $c=d$.

Now, suppose that $e$ is an edge with endpoints $u$ and $v$.  Let $G'$ be the graph obtained from $S_3$ by adding $e$.  Notice that a proper $L_3$-coloring of $S_3$ that colors $u$ and $v$ differently is also a proper $L_3$-coloring of $G'$.  This observation along with \cref{thm: examples} yields
\[ \sum_{(c,d) \in L(u) \times L(v), c \neq d} N_3(c,d) = P(G',L_3) \geq P(C_{2k+1},m) =  (m-1)^{2k+1} - (m-1).\]
Let $G''$ be the graph obtained from $G'$ by contracting $e$, and let $w$ be the vertex of $G''$ obtained via the contraction.  Let $L''$ be the $m$-assignment for $G''$ given by $L''(w) = [m]$ and $L''(x) = L_3(x)$ for each $x \in V(G'') - \{w\}$.  There is clearly a one-to-one correspondence between the proper $L_3$-colorings of $S_3$ that color $u$ and $v$ the same color and the proper $L''$-colorings of $G''$.   This observation along with \cref{thm: examples} yields
\[ \sum_{j=1}^m N_3(j,j) = P(G'', L'') \geq P(C_{2k},m) = (m-1)^{2k} + (m-1). \]
So, by \cref{lem: sum},
\begin{align*}
    P(G,L) &= \sum_{(c,d) \in L(u) \times L(v), c \neq d} \prod_{i=1}^3 N_i(c,d) + \sum_{j=1}^m \prod_{i=1}^3 N_i(j,j) \\
    & \geq \sum_{(c,d) \in L(u) \times L(v), c \neq d} (m-2)^2 N_3(c,d) + \sum_{j=1}^m (m-1)^2 N_3(j,j) \\
    & \geq P(G,m).
\end{align*}
\end{proof}

We now turn our attention to $m$-assignments $L$ for $G$ satisfying $L(u) \neq L(v)$.

\begin{lem} \label{lem: max}
Suppose $P$ is the path of length 2 with $P = u,w,v$ and $L$ is an $m$-assignment for $P$ where $m \geq 3$.  For each $c \in L(u)$ and $d \in L(v)$, let $N(c,d)$ be the number of proper $L$-colorings of $P$ where $u$ is colored with $c$ and $v$ is colored with $d$.  Suppose $a = |L(u) \cap L(v)|$.  Then, the following statements hold.
\begin{enumerate}[label=(\roman*)]
\item   For each $c \in L(u)$ and $d \in L(v)$, $m-2 \leq N(c,d) \leq m$.
\item   The number of ordered pairs $(c,d) \in L(u) \times L(v)$ for which $N(c,d) = m-2$ is at most $\frac{(a+m)^2}{4} - a$.
\end{enumerate}
\end{lem}

\begin{proof}
Since $w$ has 2 neighbors, (i) is obvious.

Now, suppose $|L(w) \cap L(u) \cap L(v)| = k_1$, $|L(w) \cap \left(L(u) - L(v)\right)| = k_2$, and $|L(w) \cap \left(L(v) - L(u)\right)| = k_3$.  Note that in order for $N(c,d) = m-2$: $c \in L(u) \cap L(w)$, $d \in L(v) \cap L(w)$, and $c \neq d$.  Consequently, the number of ordered pairs $(c,d) \in L(u) \times L(v)$ for which $N(c,d) = m-2$ is $(k_1+k_3)(k_1+k_2) - k_1.$  Since $k_2+k_3 \leq m-k_1$ and $k_1 \leq a$, we have that \[(k_1+k_3)(k_1+k_2) - k_1 \leq \left(k_1 + \frac{m-k_1}{2}\right)^2 - k_1 = \left(\frac{m+k_1}{2}\right)^2 - k_1 \leq \frac{(a+m)^2}{4} - a.\] \end{proof}

Our next lemma requires the DP color function which we now introduce.  As it is introduced, graph $G$ is not necessarily a copy of $\Theta(2,2,2k)$ for some $k \geq 2$ but rather an arbitrary graph.  The concept of DP-coloring was first put forward in 2015 by Dvo\v{r}\'{a}k and Postle under the name \emph{correspondence coloring} (see~\cite{DP15}).  Intuitively, DP-coloring generalizes list coloring by allowing the colors that are identified as the same to vary from edge to edge.  Formally, for a graph $G$, a \emph{DP-cover} (or simply a \emph{cover}) of $G$ is an ordered pair $\mathcal{H}=(L,H)$, where $H$ is a graph and $L:V(G)\to 2^{V(H)}$ is a function satisfying the following conditions: 
    \begin{itemize}
		\item $\{L(v) : v \in V(G)\}$ is a partition of $V(H)$ into $|V(G)|$ parts, 

		\item for every pair of adjacent vertices $u$, $v\in V(G)$, the edges in $E_H\left(L(u),L(v)\right)$ form a matching (not necessarily perfect and possibly empty), and

		\item $\displaystyle E(H) = \bigcup_{uv \in E(G)} E_{H}(L(u),L(v)).$
    \end{itemize}
    
Suppose $\mathcal{H}=(L,H)$ is a cover of a graph $G$.  A \emph{transversal} of $\mathcal{H}$ is a set of vertices $T\subseteq V(H)$ containing exactly one vertex from each $L(v)$. A transversal $T$ is said to be \emph{independent} if $T$ is an independent set in $H$.  If $\mathcal{H}$ has an independent transversal $T$, then $T$ is said to be a \emph{proper $\mathcal{H}$-coloring} of $G$, and $G$ is said to be \emph{$\mathcal{H}$-colorable}.  A \emph{$k$-fold cover} of $G$ is a cover $\mathcal{H}=(L,H)$ such that $|L(v)|=k$ for all $v\in V(G)$.

Suppose $\mathcal{H} = (L,H)$ is a cover of graph $G$.  We let $P_{DP}(G, \mathcal{H})$ be the number of proper $\mathcal{H}$-colorings of $G$.  Then, the \emph{DP color function}, denoted $P_{DP}(G,m)$, is the minimum value of $P_{DP}(G, \mathcal{H})$ where the minimum is taken over all possible $m$-fold covers $\mathcal{H}$ of $G$.  

Now, suppose that $L$ is an $m$-assignment for $G$.  The \emph{cover of $G$ corresponding to $L$}, denoted $\mathcal{H}_L = (\Lambda_L,H_L)$, is the cover of $G$ defined as follows.  For each $v \in V(G)$, $\Lambda_L(v) = \{(v,c) : c \in L(v) \}$, and $H_L$ is the graph with vertex set $\bigcup_{v \in V(G)} \Lambda_L(v)$ and edges created so that for any $(u,c_1),(v,c_2) \in V(H_L)$, $(u,c_1)(v,c_2) \in E(H)$ if and only if $uv \in E(G)$ and $c_1 = c_2$.  Notice that if $\mathcal{C}$ is the set of proper $L$-colorings of $G$ and $\mathcal{T}$ is the set of proper $\mathcal{H}$-colorings of $G$, then the function $h: \mathcal{C} \rightarrow \mathcal{T}$ given by $h(f) = \{(v,f(v)) : v \in V(G) \}$ is a bijection.  So, for any graph $G$ and $m \in \N$,
\[P_{DP}(G, m) \leq P_\ell(G,m) \leq P(G,m).\]

\begin{lem}\label{lem: even path}
Let $P=x_1,x_2,\dots, x_{2k+1}$ where $k \in \N$ and $L$ be an $m$-assignment of $P$ with $m \geq 3$.  Suppose $c \in L(x_1)$ and $d \in L(x_{2k+1})$, and let $N(c,d)$ be the number of proper $L$-colorings of $P$ where $x_1$ is colored with $c$ and $x_{2k+1}$ is colored with $d$. For each $(c,d)\in L(x_1)\times L(x_{2k+1})$, \[\frac{(m-1)^{2k+1}-(m-1)}{m(m-1)}\le N(c,d).\]
Furthermore, there exist at least $m$ elements $(y,z) \in L(x_1)\times L(x_{2k+1})$ such that 
$$\frac{(m-1)^{2k+1}-(m-1)}{m(m-1)}+1=\frac{(m-1)^{2k}+(m-1)}{m}\le N(y,z).$$ 
\end{lem}

\begin{proof}
Let $H'$ be the graph obtained from $H_L$ by arbitrarily adding edges so that \\ $E_{H'}\left(\Lambda_L(x_i),\Lambda_L(x_{i+1})\right)$ is a perfect matching for each $i \in [2k]$.  Note that $H'$ is the disjoint union of $m$ paths. Then, let $\mathcal{H}' = (\Lambda_L,H')$.  Clearly, $\mathcal{H}'$ is an $m$-fold cover of $P$.  Furthermore, let $\mathcal{T}$ be the set of proper $\mathcal{H}'$-colorings of $P$, and let $\mathcal{C}$ be the set of proper $L$-colorings of $P$.  Then we can injectively map each $T \in \mathcal{T}$ to an $f_T \in \mathcal{C}$ by letting $f_T(v)$ be the second coordinate of the ordered pair in $T$ with first coordinate $v$ for each $v \in V(G)$.

Lemma~8 in~\cite{BK23} implies that $\mathcal{T}$ has at least $((m-1)^{2k+1}-(m-1))/(m(m-1))$ elements that contain $(x_1,c)$ and $(x_{2k+1},d)$.  It further says that there are $((m-1)^{2k}+(m-1))/m$ elements of $\mathcal{T}$ that contain $(x_1,c)$ and $(x_{2k+1},d)$ if and only if there is a path in $H'$ connecting $(x_1,c)$ and $(x_{2k+1},d)$.  The result immediately follows. 
\end{proof}

\begin{lem}\label{lem: AMGM3}
Suppose $G = \Theta(2,2,2k)$ with $k \geq 2$. Suppose that $L$ is a $3$-assignment for $G$ with $m= 3$ such that $L(u)\ne L(v)$.  Then 
$P(G,L)\ge P(G,3)$.
\end{lem}
\begin{proof}
    Note that $P(G,3)=4(2^{2k}+2)+2^{2k+1}-2$. Let $a=|L(u)\cap L(v)|$, and notice that $a \in \{0,1,2\}$ since $L(u) \neq L(v)$.
 \cref{lem: max} implies that we have that $N_i(c,d)=1$ for $i\in [2]$ for at most 4 pairs $(c,d)\in L(u)\times L(v)$.  
    Consequently,
    \[\prod_{(c,d)\in L(u)\times L(v)} N_i(c,d)\ge (1)^4(2)^5.\]
    \cref{lem: even path} implies that 
    \[\prod_{(c,d)\in L(u)\times L(v)}N_3(c,d)\ge \left(\frac{2^{2k}-1}{3}\right)^6\left(\frac{2^{2k}+2}{3}\right)^3.\]
    

    By \cref{cor: amgmpgl}, we have that
    \[P(G,L)\ge 9\left(1^8\cdot 2^{10} \left(\frac{2^{2k}-1}{3}\right)^6\left(\frac{2^{2k}+2}{3}\right)^3\right)^{1/9}=6\left(2(2^{2k}-1)^6(2^{2k}+2)^3\right)^{1/9}.\]
    Let $x=2^{2k}-1$.  Then $P(G,3)=4(x+3)+2x=6x+12$ and \[P(G,L)\ge6\left(2x^6 (x+3)^3\right)^{1/9}=6x\left(2\left(1+\frac{3}{x}\right)^3\right)^{1/9}.\]

    To complete the proof, we will show that $6x\left(2\left(1+\frac{3}{x}\right)^3\right)^{1/9}\ge 6x+12$.
    When $k=2$, $6(15)\left(2\left(1+\frac{3}{15}\right)^3\right)^{1/9}-6(15) > 12$, as desired.  
    For $k\ge 3$, 
    \[6x\left(2\left(1+\frac{3}{x}\right)^3\right)^{1/9}-6x>6x(2^{1/9}-1)\ge 6(63)(2^{1/9}-1)>12.\]
    Hence $P(G,L)\ge P(G,3)$.
\end{proof}

\begin{lem} \label{lem: DPconnection}
Suppose $G$ is an arbitrary graph and $L$ is an $m$-assignment for $G$.  Suppose $uv \in E(G)$, $|L(u)-L(v)| = d \geq 1$, and for any $x \in L(u)$ and $y \in L(v)$ with $x \neq y$, there are at least $C$ proper $L$-colorings of $G$ that color $u$ with $x$ and $v$ with $y$.  Then,
\[P(G,L) \geq P_{DP}(G,m) + Cd.\]
\end{lem}

\begin{proof}
Let $L(u)-L(v) = \{c_1, \ldots, c_d\}$.  Let $H'$ be the graph obtained from $H_L$ by arbitrarily adding edges so that $E_{H'}(\Lambda_L(w),\Lambda_L(z))$ is a perfect matching whenever $wz \in E(G)$.  Then, let $\mathcal{H}' = (\Lambda_L,H')$ so that $\mathcal{H}'$ is an $m$-fold cover of $G$.  As in the proof of \cref{lem: even path}, when $\mathcal{T}$ is the set of proper $\mathcal{H}'$-colorings of $G$ and $\mathcal{C}$ is the set of proper $L$-colorings of $G$, we can injectively map each $T \in \mathcal{T}$ to an $f_T \in \mathcal{C}$.

Now, suppose that for each $i \in [d]$, $N_{H'}((u,c_i)) \cap \Lambda_L(v) = \{(v,d_i)\}$.  It follows that no $T \in \mathcal{T}$ is mapped to an $f \in \mathcal{C}$ satisfying $f(u) = c_i$ and $f(v) = d_i$ for some $i \in [d]$.  Since for each $i \in [d]$ there are at least $C$ elements of $\mathcal{C}$ that color $u$ with $c_i$ and $v$ with $d_i$,
\[P(G,L) = |\mathcal{C}| \geq |\mathcal{T}|+Cd \geq P_{DP}(G,m) + Cd.\]
\end{proof}

\begin{lem} \label{lem: greed}
Suppose $G = \Theta(2,2,2k)$ with $k \geq 2$.  Suppose that $L$ is an $m$-assignment for $G$ with $m \geq 4$.  If there is a $wz \in E(G)$ such that there is an $x \in L(w)-L(z)$ and a $y \in L(z)-L(w)$, then there are at least $(m-1)^{2k-1}(m-2)^2$ proper $L$-colorings of $G$ that color $w$ with $x$ and $z$ with $y$.
\end{lem}

\begin{proof}
We may assume without loss of generality that $uv$ is in the copy of $C_{2k+2}$, $W$, contained in $G$ that consists of the paths $S_2$ and $S_3$.  We can greedily construct a proper $L$-coloring of $G$ with the desired property as follows. Begin by first coloring $w$ with $x$ and $z$ with $y$.  Then greedily color the remaining $2k$ vertices of $W$ (by proceeding around the cycle from $z$).  Finally, color $x_1$ with a color in $L(x_1)$ that is not in conflict with the color assigned to either of the neighbors of $x_1$.

Notice that in constructing this greedy coloring we have one choice in how we color our first two vertices (i.e., $w$ and $z$).  We have at least $(m-1)$ color choices for each of the next $(2k-1)$ vertices, and we have at least $(m-2)$ color choices for each of our last two vertices.  The result immediately follows.
\end{proof}

Before proving the next lemma, we note that by Theorem~6 in~\cite{BK23}, $P_{DP}(\Theta(2,2,2k),m) = ((m-1)^{2k+4}-(m-1)^{2k}-2(m-1)^2+2)/m$.

\begin{lem} \label{lem: allbut3}
Suppose $G = \Theta(2,2,2k)$ with $k \geq 2$.  Suppose that $L$ is an $m$-assignment for $G$ with $m \geq 4$ and $L(u) \neq L(v)$.  Then, $P(G,L) \geq P(G,m)$.
\end{lem}

\begin{proof}
Since $L(u) \neq L(v)$, there is a $wz \in E(G)$ such that there is an $x \in L(w)-L(z)$ and a $y \in L(z)-L(w)$.  So, by \cref{lem: DPconnection,lem: greed}
\[P(G,L) \geq P_{DP}(G,m)+ (m-1)^{2k-1}(m-2)^2.\]

It is easy to see that $P(G,m)-P_{DP}(G,m) = (m-1)^{2k} + 2(m-1)^2 + (m-1) - 2$.  Now, note that since $m \geq 4$ and $k \geq 2$,
\[2(m-1)^2 + (m-1) - 2 \leq 3(m-1)^2 \leq (m-1)^3 \leq (m-4)(m-1)^{2k} + (m-1)^{2k-1}.\]
This means $(m-1)^{2k} + 2(m-1)^2 + (m-1) - 2 \leq (m-1)^{2k-1}(m-2)^2$.  So, $P(G,L) \geq P(G,m)$.
\end{proof}

\thmthreeandabove*

\begin{proof}
Suppose $m \geq 3$ is fixed and that $L$ is an $m$-assignment for $G$ such that $P(G,L) = P_{\ell}(G,m)$.  Clearly, $P(G,m) \geq P(G,L)$.  If $L(u)=L(v)$, then \cref{lem: samelist} implies $P(G,L) \geq P(G,m)$ and the result follows.  In the case that $m=3$ and $L(u) \neq L(v)$, \cref{lem: AMGM3} implies $P(G,L) \geq P(G,m)$ and the result follows.  In the case that $m > 3$ and $L(u) \neq L(v)$, \cref{lem: allbut3} implies $P(G,L) \geq P(G,m)$ and the result follows.    
\end{proof}

\section{Support for \cref{conj: join}} \label{prob}

In this section, in support of \cref{conj: join}, we prove $K_1 \vee \Theta(2,2,2k)$ is enumeratively chromatic-choosable (\cref{thm: thetagraph}).  We also give an example of a graph with chromatic number $3$ that is chromatic-choosable but is not weakly enumeratively chromatic-choosable which answers \cref{ques: notecc} in the affirmative for $k=3$.

\begin{lem} \label{lem: prob}
Suppose $G$ is a graph with $V(G) = \{v_1, \ldots, v_n \}$.  Suppose that $m \in \N$, $d_1, \ldots, d_n$ are nonnegative integers, and $L$ is a list assignment for $G$ such that $|L(v_i)|=m+d_i$.  Then,
\[P(G,L) \geq \left \lceil \frac{P_{\ell}(G,m)\prod_{i=1}^n (m+d_i)}{m^n} \right \rceil\]
\end{lem}

\begin{proof}
Let $\mathcal{F}$ be the set of all proper $L$-colorings of $G$.  For each $i \in [n]$ uniformly at random delete $d_i$ colors from $L(v_i)$ and call the resulting $m$-assignment for $G$, $L'$.  Clearly, $P_{\ell}(G,m) \leq P(G, L')$.

Now, for each $f \in  \mathcal{F}$, let $X_f$ be the random variable that is equal to 1 if $f$ is a proper $L'$-coloring of $G$ and is equal to 0 otherwise.  Since every proper $L'$-coloring of $G$ is also a proper $L$-coloring of $G$, we have that $P(G,L') = \sum_{f \in \mathcal{F}} X_f$.  Notice that for each $f \in \mathcal{F}$ the probability that $f(v_i) \in L'(v_i)$ is \[1 - \frac{d_i}{m+d_i} = \frac{m}{m+d_i}\] for each $i \in [n]$.  
Consequently, the probability that $X_f = 1$ is $\prod_{i=1}^n m/(m+d_i)$.  This means that
\[\E \left( \sum_{f \in \mathcal{F}} X_f \right) = |\mathcal{F}| \prod_{i=1}^n \frac{m}{m+d_i} = \frac{m^nP(G,L)}{\prod_{i=1}^n (m+d_i)}.  \]
So, there is an $m$-assignment, $L''$ of $G$ formed by deleting $d_i$ colors from $L(v_i)$ for each $i \in [n]$ such that
\[P_{\ell}(G,m) \leq P(G,L'') \leq \frac{m^nP(G,L)}{\prod_{i=1}^n (m+d_i)}.\]
The fact that $P(G,L)$ is an integer then implies the desired result.
\end{proof}

Before proving \cref{thm: thetagraph}, we need an observation and proposition.

\begin{obs} \label{obs: star}
Suppose $G = K_{1,n}$ where $n \in \N$, and $L$ is a 2-assignment for $G$.  If $L$ does not assign the same list of colors to each $v \in V(G)$, $P(G,L) \geq 3$.    
\end{obs}

\begin{pro} \label{pro: joinbipartite}
Suppose $G$ is a connected graph with $\chi(G)=2$ and bipartition $X$, $Y$.  If $P_{\ell}(G,2) \geq 1$, then $K_1 \vee G$ is weakly enumeratively chromatic-choosable.    
\end{pro}

\begin{proof}
In the case $P_{\ell}(G,2) \geq 2$, \cref{thm: joinK1} implies the desired result.  So, we suppose that $P_{\ell}(G,2) = 1$.

Let $H = K_1 \vee G$, and suppose that $u$ is the vertex from the copy of $K_1$ used to form $H$.  Clearly, $P(H,3)=6$.  Suppose that $L$ is a 3-assignment for $H$ such that $P_{\ell}(H,3) = P(H,L)$.  We will show that $P(H,L) \geq 6$.

Suppose that $L(u) = \{c_1, c_2, c_3 \}$.  For each $i \in [3]$, let $o_i = |\{v \in V(G) : c_i \notin L(v) \}|$.  We will show $P(H,L) \geq 6$ when: (1) $o_i=0$ for some $i \in [3]$ and (2) $o_i \geq 1$ for each $i \in [3]$.

First, suppose without loss of generality that $o_1 = 0$.  Let $L_1$ be the $2$-assignment for $H[X \cup Y]$ given by $L_1(v) = L(v) - \{c_1\}$ for each $v \in X \cup Y$.  Let $L_2$ be the $2$-assignment for $H[Y \cup \{u\}]$ given by $L_2(v) = L(v) - \{c_1\}$ for each $v \in Y \cup \{u\}$.  Let $L_3$ be the $2$-assignment for $H[X \cup \{u\}]$ given by $L_3(v) = L(v) - \{c_1\}$ for each $v \in X \cup \{u\}$.

Notice that if $L_2$ and $L_3$ assign the same list to every vertex in $H[Y \cup \{u\}]$ and $H[X \cup \{u\}]$ respectively, then $L(v) = \{c_1,c_2,c_3\}$ for each $v \in V(G)$ which immediately implies $P(H,L) = P(H,3)=6$.  So, we may assume without loss of generality that $L_2$ doesn't assign the same list to every vertex in $H[Y \cup \{u\}]$.  Clearly, a proper $L_1$-coloring, a proper $L_2$-coloring, and a proper $L_3$-coloring can be extended to a proper $L$-coloring of $H$ by coloring $u$ with $c_1$, coloring each vertex in $X$ with $c_1$, and coloring each vertex in $Y$ with $c_1$ respectively.  Then, \cref{thm: examples} and \cref{obs: star} imply
\begin{align*}
P(H,L) &\geq P(H[X \cup Y],L_1) + P(H[Y \cup \{u\}],L_2) + P(H[X \cup \{u\}],L_3) \\ 
&\geq P_{\ell}(G,2) + 3 + P_{\ell}(K_{1,|X|},2) \\
&\geq 1+3+2=6.    
\end{align*}
Now, we may suppose $o_i \geq 1$ for each $i \in [3]$.  For each $i \in [3]$, let $L^{(i)}$ be the list assignment for $G$ given by $L^{(i)}(v) = L(v) - \{c_i\}$ for each $v \in V(G)$.  Since $o_i \geq 1$ for each $i \in [3]$, \cref{lem: prob} implies $P(G,L^{(i)}) \geq \lceil 1(3)/2 \rceil = 2$.  Thus,
$$P(H,L) \geq \sum_{i=1}^3 P(G,L^{(i)}) \geq 6$$
as desired.
\end{proof}

\begin{thm} \label{thm: thetagraph}
If $G = K_1 \vee \Theta(2,2,2k)$ where $k \in \N$, then $G$ enumeratively chromatic-choosable.
\end{thm}

\begin{proof}
When $k=1$, the result immediately follows from \cref{thm: examples}.

Now, suppose that $k \geq 2$.  By~\cref{cor: Theta222tcc} and \cref{pro: joinbipartite}, we have that $P_{\ell}(G,3) = P(G,3)$.  Then, by \cref{thm: joinK1,thm: 3andabove}, we have that $P_{\ell}(G,m) = P(G,m)$ whenever $m \geq 4$.  Thus, $G$ is enumeratively chromatic-choosable.
\end{proof}

We end this section by presenting an example of a graph with chromatic number $3$ that is chromatic-choosable but is not weakly enumeratively chromatic-choosable which answers \cref{ques: notecc} in the affirmative for $k=3$.

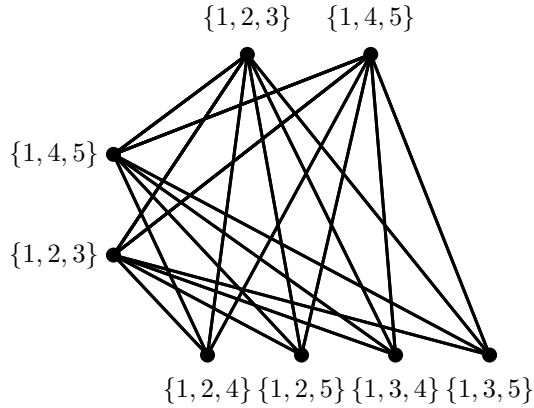
\begin{figure}[htb] 
 \begin{center}
     \begin{tikzpicture}
        [scale=1,auto=left,every node/.style={circle, fill, inner sep=0 pt, minimum size=2mm, outer sep=0pt},line width=.4mm]
        \node(x1) at (0,1.33)[label={[label distance=2pt]left:{\small$\{1,2,3\}$}}] {};
        \node (x2) at (0,2.67)[label={[label distance=2pt]left:{\small$\{1,4,5\}$}}] {};
        \node (z1) at (1.25,0)[label={[label distance=-7pt]below:{\small$\{1,2,4\}$}}]{};
        \node (z2) at (2.5,0)[label={[label distance=-7pt]below:{\small$\{1,2,5\}$}}]{};
        \node (z3) at (3.75,0)[label={[label distance=-7pt]below:{\small$\{1,3,4\}$}}]{};
        \node (z4) at (5,0)[label={[label distance=-7pt]below:{\small$\{1,3,5\}$}}]{};
        \node (y1) at (1.78,4)[label={[label distance=-7pt]above:{\small$\{1,2,3\}$}}] {};
        \node (y2) at (3.42,4)[label={[label distance=-7pt]above:{\small$\{1,4,5\}$}}]{};
        \begin{scope}[on background layer]
        \foreach \i in {1,2}{
        \foreach \k in {1,2}{
        \foreach \j in {1,2,3,4}{
        \draw[line width=.4mm] (x\i.center) to (y\k.center);
        \draw[line width=.4mm] (x\i.center) to (z\j.center);
        \draw[line width=.4mm] (y\k.center) to (z\j.center); }}}
        \end{scope}
        \end{tikzpicture}
        \end{center}
        \caption{The graph $G=K_{2,2,4}$ with its list assignment $L$.}
        \label{fig:k224}
\end{figure}

\begin{pro} \label{pro: counterexample}
Suppose $G = K_{2,2,4}$.  Then, $G$ is chromatic-choosable, but $G$ is not weakly enumeratively chromatic-choosable.
\end{pro}

\begin{proof}
Suppose $G=K_{2,2,4}$ with partite sets $X = \{x_1, x_2\}$, $Y= \{y_1, y_2\}$, and \\ $Z = \{z_1,z_2,z_3,z_4\}$. In~\cite{EO02} it is shown that $\chi_{\ell}(G)=3$ which implies that $G$ is chromatic-choosable.

To show that $G$ is not enumeratively chromatic-choosable, we must show $P_{\ell}(G,3) < 6 = P(G,3)$.  Suppose $L$ is the 3-assignment for $G$ given by: $L(x_1)=L(y_1) = [3]$, $L(x_2)=L(y_2) = \{1,4,5\}$, $L(z_1) = \{1,2,4\}$, $L(z_2) = \{1,2,5\}$, $L(z_3) = \{1,3,4\}$, and $L(z_4) = \{1,3,5\}$ (see \cref{fig:k224}).

Now, suppose that $f$ is a proper $L$-coloring of $G$.  We first observe that $f(x) \neq 1$ for each $x \in X \cup Y$.  Consequently, we know that $\{2,3\} \not \subseteq f(Z)$ and $\{4,5\} \not \subseteq f(Z)$.  Since $x_1y_1, x_2y_2 \in E(G)$, we have that $\{1,z\} \not \subseteq f(Z)$ whenever $z \in \{2,3,4,5\}$.  This means $|f(Z)| = 1$ which immediately implies that $f(Z) = \{1\}$.  Then, we have that $f(\{x_1,y_1\}) = \{2,3\}$ and $f(\{x_2,y_2\}) = \{4,5\}$.  It follows that there are only four possible rules for $f$.

This means $4 = P(G,L) \geq P_{\ell}(G,3)$ as desired.
\end{proof}

{\bf Acknowledgment.}  The authors would like to thank Doug West for suggesting the study of enumeratively chromatic-choosable graphs. 

	\bibliographystyle{hplain}
\bibliography{bibliography}{}
\vspace{-5mm}

\end{document}